\documentclass{article}

\usepackage{
  amsmath, 
  amsthm, 
  amssymb,
}

\usepackage[
]{enumitem}

\newtheorem{theorem}{Theorem}

\newtheorem{corollary}[theorem]{Corollary}

\theoremstyle{remark}

\theoremstyle{definition}
\newtheorem{example}{Example}   

\DeclareMathOperator{\supp}{supp}

\PassOptionsToPackage{hyphens}{url} 
\usepackage[
colorlinks=true,
]{hyperref} 

\title{Lobachevsky-type Formulas via Fourier Analysis}
\author{\shortstack{Runze Cai$^1$\\cairunze@sjtu.edu.cn} \and \shortstack{Horst Hohberger$^1$\\horst@sjtu.edu.cn} \and \shortstack{Mian Li$^{1}$\\mianli@sjtu.edu.cn}}
\date{%
    $^1$University of Michigan-Shanghai Jiao Tong University Joint Institute, Shanghai Jiao Tong University, Minhang, Shanghai 200240, China\\[\baselineskip]%
    \today
}
\begin{document}

\maketitle

\begin{abstract}
Recently renewed interest in the Lobachevsky-type integrals and interesting identities involving the cardinal sine motivate an 
extension of the classical Parseval formula involving both periodic and non-periodic functions. We develop a version of the Parseval formula  
that is often more practical in applications and illustrate its use by extending recent results on Lobachevsky-type integrals. Some previously known, interesting identities are re-proved in a more transparent manner and new formulas for integrals involving 
cardinal sine and Bessel functions are given.
\end{abstract}

\section{Introduction}
\label{sec:introduction}

The following is known as a Lobachevsky-type integral: 
\begin{equation*}
  \int_{-\infty}^{\infty} \left(\frac{\sin\pi x}{\pi x} \right)^{k}p(x)\, dx
\end{equation*}
Here $k\in\mathbb{N}\setminus\{0\}$, and $p\colon\mathbb{R}\to\mathbb{C}$ is a  periodic, real-valued function with period $T>0$ that is assumed to be integrable over a single period. 
Recently, Jolany~\cite{jolany_extension_2018} has published identities for this integral when $k$ is even, with continuous $p$ being of period $T=1$, using methods of complex analysis. We will base our discussion on the Fourier transform and obtain corresponding identities for all $k\in\mathbb{N}\setminus\{0\}$ and  $p$ integrable of arbitrary period $T$.

The Lobachevsky integral is closely related to the Shannon basis of information theory. 
Much of our treatment is inspired by identities that are ``folklore'' in the signal processing community, where the reconstruction formula in the Shannon basis is precisely the cardinal sine expansion~\cite{unser_sampling-50_2000}.

\section{Parseval Formula}
\label{sec:parseval}

For functions $f\colon\mathbb{R}\to\mathbb{C}$, we define the Fourier transform as follows,
\begin{equation}
  \label{eq:fourier}
  \hat{f}(\xi)=(\mathcal{F}f)(\xi)=\int_{-\infty}^{\infty}f(x)e^{-2\pi i x\xi}\, dx,
\end{equation}
whenever the integral exists. Similarly, the inverse Fourier transform is defined by 
\begin{equation*}
\check{f}(\xi)=(\mathcal{F}^{-1}f)(\xi)=\int_{-\infty}^{\infty}f(x)e^{2\pi i x\xi}\, dx,
\end{equation*}
again, whenever the integral exists. 

The set of absolutely integrable functions on the real axis is denoted by $L^{1}(\mathbb{R})$, while 
$\operatorname*{BV}(\mathbb{R})$ denotes the set of functions that are of bounded variation on $\mathbb{R}$.

In this treatment, complex-valued, periodic functions of period $T>0$ that are absolutely integrable over a single period play a major role. The set of such functions is denoted by $L^{1}([-\frac{T}{2}, \frac{T}{2}])$. In the case $T=1$, we will write simply $L^{1}(\mathbb{T})$. 

For $f\in L^{1}([-\frac{T}{2}, \frac{T}{2}])$ we define the Fourier coefficient
\begin{align*}
\hat{f}(n)&=\int_{-T/2}^{T/2}e^{-2\pi i n x/T}f(x)\, dx, & n&\in\mathbb{Z}.
\end{align*}
The minor clash of notation with~\eqref{eq:fourier} should not give rise to confusion~\cite{stein_fourier_2011}. 

We will also use the standard notation
\begin{align*}
f(x^-)&:=\lim_{\varepsilon\searrow0}f(x-\varepsilon), & f(x^+)&:=\lim_{\varepsilon\searrow0}f(x+\varepsilon)	
\end{align*}
for any function $f$ on $\mathbb{R}$ and $x\in\mathbb{R}$.

The classical strong form Parseval formula~\cite[Theorem~8.18, Chapter~IV]{zygmund_trigonometric_2002} of the so-called ``mixed type'' (i.e., a periodic and a non-periodic function) may be formulated as follows,

\begin{theorem}[Parseval formula of mixed type in the strong form]
\label{thm:4}
Let $f\in L^1(\mathbb{T})$ and $g\in L^{1}(\mathbb{R})\cap \operatorname*{BV}(\mathbb{R})$. 
Then
\begin{equation}
  \label{eq:89}
  \int_{-\infty}^{\infty}\overline{f(x)}g(x)\, dx=\sum_{n=-\infty}^{\infty}\overline{\hat{f}(n)}\hat{g}(n)
\end{equation}
where $\overline{f(x)}$ denotes the complex conjugate of $f(x)$.
\end{theorem}

The drawback of this theorem is that the condition on $g$ is frequently difficult to check: it may not be easy to show 
that $g$ is of bounded variation. More significantly, in certain interesting situations, e.g., where the cardinal sine 
is involved, Theorem~\ref{thm:4} is simply not applicable. 

Therefore, we establish the following result in the spirit of~\cite[Theorem~47]{titchmarsh_introduction_1948}, which 
yields a ``weak form'' Parseval formula of mixed type under the condition that the 
Fourier transform of $g$, rather than $g$ itself, is of compact support and of bounded variation at appropriate points. 
We denote the support of a function $g$ by $\operatorname*{supp}{g}$.


\begin{theorem}[Parseval formula of mixed type in the weak form]
\label{thm:2}
Let $f\in L^1(\mathbb{T})$, $g\in L^{1}(\mathbb{R})$, and suppose that there exists some $A>0$ such that 
$\supp g\subset[-A,A]$. Further, let $g$ be of bounded variation in neighborhoods of all $n\in\mathbb{Z}$ 
with $\lvert{n}\rvert\leq A$. Then
\begin{equation*}
\int_{-\infty}^{\infty}f(x)\hat{g}(x)\, dx=\sum_{\substack{n\in\mathbb{Z} \\ \lvert{n}\rvert\leq A}}\hat{f}(n)\cdot\frac{g(n^{-})+g(n^{+})}{2}
\end{equation*}
\end{theorem}

Theorem~\ref{thm:2} can be adapted to periodic functions $p$ with arbitrary period $T>0$  by setting $f(x):=p(Tx)$, yielding 

\begin{corollary}
  \label{cor:2}
Let $p\in L^1([-\frac{T}{2}, \frac{T}{2}])$, $g\in L^{1}(\mathbb{R})$, and suppose that there exists some $A>0$ such that 
$\supp g\subset[-A,A]$. Further, let $g$ be of bounded variation in neighborhoods of all points $n/T$, $n\in\mathbb{Z}$, 
with $\lvert{n/T}\rvert\leq A$. Then
\begin{equation}
\label{eq:27}
\int_{-\infty}^{\infty}p(x)\hat{g}(x)\, dx=\sum_{\substack{n\in\mathbb{Z} \\ \lvert{n/T}\rvert\leq A}}
\hat{p}(n)\cdot\frac{g((n/T)^{-})+g((n/T)^{+})}{2}
\end{equation}
\end{corollary}

Functions of compact support are related to signals that are ``band-limited'' in the parlance of the signal processing community, since it is possible to recover a band-limited (continuous) signal by appropriate (discrete) sampling~\cite{unser_sampling-50_2000}. 

\begin{example}
Consider $\psi_1\in L^1(\mathbb{R})$ given by
\begin{equation*}
  \psi_{1}(x)=
  \begin{cases}
    \sqrt{1-x^{2}}, &-1\leq x\leq1. \\
    0, &\textup{otherwise}
  \end{cases}
\end{equation*}
Its Fourier transform is given by
\begin{equation*}
  \hat{\psi}_{1}(\xi)=
  \begin{cases}
    \dfrac{J_{1}(2\pi\xi)}{2\xi}, &x\neq0 \\[2ex]
    \dfrac{\pi}{2}, &x=0
  \end{cases}
\end{equation*}
where $J_{1}$ is the Bessel function of the first kind of order one. (The function given by $J_{1}(\xi)/\xi$ and its scaled versions are sometimes called Sombrero function, besinc function, or jinc function.)  Parseval's formula 
(either \eqref{eq:27} or \eqref{eq:89}) then gives
\begin{equation*}
    \int_{-\infty}^{\infty}\frac{J_{1}(2\pi x)}{2 x}p(x)\, dx
=\sum_{\lvert{n/T}\rvert<1}\hat{p}(n)\sqrt{1-(n/T)^{2}}
\end{equation*}
for any $p\in L^{1}([-\frac{T}{2}, \frac{T}{2}])$. If $0< T\leq1$, only the summand for index $n=0$ remains and we have
\begin{equation*}
  \int_{-\infty}^{\infty}\frac{J_{1}(2\pi x)}{2 x}p(x)\, dx=\hat{p}(0)=\int_{-T/2}^{T/2}p(x)\, dx.
\end{equation*}
To apply Corollary~\ref{cor:2} we need to check that $\psi_1$ is of bounded variation at least locally near 
any point of $\mathbb{R}$, which is not difficult. On the other hand, invoking Theorem~\ref{thm:4} would entail verifying 
that $\hat{\psi}_1$ is of bounded variation, a much more difficult task. 
\end{example}

\begin{example}
Now consider $\psi_2\in L^1(\mathbb{R})$ given by
\begin{equation*}
  \psi_{2}(x)=
  \begin{cases}
    \dfrac{1}{\sqrt{1-x^{2}}}, &-1<x<1 \\[2ex]
    0, &\textup{otherwise}
  \end{cases}
\end{equation*}
with Fourier transform
\begin{equation*}
  \hat{\psi}_{2}(\xi)=\pi J_{0}(2\pi \xi)
\end{equation*}
where $J_{0}$ is the Bessel function of the first kind of order zero. Since $J_{0}\not\in L^{1}(\mathbb{R})$, 
Theorem~\ref{thm:4} can not be applied.

Observe that $\supp\psi_{2}=[-1,1]$ and that $\psi_{2}$ is of bounded variation in the neighborhoods of any point 
except $\pm 1$. Then for $T\not\in\mathbb{N}$, we can apply Corollary~\ref{cor:2} to 
$p\in L^{1}([-\frac{T}{2}, \frac{T}{2}])$ and obtain
\begin{equation*}
\pi\int_{-\infty}^{\infty}J_{0}(2\pi x)p(x)\, dx 
=\sum_{\lvert{n/T}\rvert<1}\frac{\hat{p}(n)}{\sqrt{1-(n/T)^{2}}}
\end{equation*}
As before, if $0<T<1$, we have
\begin{equation*}
\pi\int_{-\infty}^{\infty}J_{0}(2\pi x)p(x)\, dx=\hat{p}(0)=\int_{-T/2}^{T/2}p(x)\, dx.
\end{equation*}
\end{example}

\section{Lobachevsky Integral Formulas}
\label{sec:lobachevsky}

We define the usual convolution of $f,g\in L^1(\mathbb{R})$ by
\begin{equation*}
  (f*g)(x):=\int_{-\infty}^{\infty}f(y)g(x-y)\, dy
\end{equation*}
and write 
\begin{equation*}
  f^{*k}:=\underbrace{f*f*\cdots*f}_{k\textup{ times}}.
\end{equation*}
We introduce the real function $\Pi$ given by
\begin{equation*}
  \Pi(x):=
  \begin{cases}
    1, & \left| x \right|<1/2, \\
    1/2, &  x =\pm1/2, \\
    0, & \left| x \right|>1/2,
  \end{cases}
\end{equation*}
as well as the normalized cardinal sine on $\mathbb{R}$,
\begin{equation*}
  \operatorname*{sinc}(x):=
  \begin{cases}
    \dfrac{\sin\pi x}{\pi x}, &x\neq0 \\[2ex]
    1, &x=0.
  \end{cases}
\end{equation*}
We remark that $\operatorname*{sinc}=\mathcal{F}\Pi$,  
and, more generally, $\operatorname*{sinc}^{k}=\mathcal{F}(\Pi^{*k})$ for $k\in\mathbb{N}\setminus\{0\}$. Furthermore, note that $\operatorname*{supp}{\Pi^{*k}}=[-\frac{k}{2}, \frac{k}{2}]$, and $\Pi^{*k}$ is piecewise polynomial, hence also in $L^{1}(\mathbb{R})$ and is of bounded variation in neighborhoods of all points in $\operatorname*{supp}{\Pi^{*k}}$. Note that $\Pi^{*k}$ are also known as the ``centered B-splines'' in the signal processing community~\cite{unser_sampling-50_2000}.

As a special case of Corollary~\ref{cor:2}, we have the following theorem on the Lobachevsky integral formula. 
\begin{theorem}[]
\label{thm:1}
Let $p\in L^1([-\frac{T}{2}, \frac{T}{2}])$ for some $T>0$ and $k\in\mathbb{N}\setminus\{0\}$. Then 
\begin{equation*}
  \int_{-\infty}^{\infty}\operatorname{sinc}^{k}(x)p(x)\, dx=\sum_{|n/T|\leq k/2}\hat{p}(n)\Pi^{*k} \left( n/T\right)
\end{equation*}
If $k\geq2$, the range of summation may be reduced as follows,
\begin{equation*}
   \int_{-\infty}^{\infty}\operatorname{sinc}^{k}(x)p(x)\, dx=\sum_{|n/T|<k/2}\hat{p}(n)\Pi^{*k}(n/T).
\end{equation*}
\end{theorem}

\begin{corollary}[]
Let $k\in\mathbb{N}\setminus\{0\}$ and $p\in L^1([-\frac{T}{2}, \frac{T}{2}])$ for some $T>0$ with $kT\leq2$ if $k\geq 2$, or 
$0<T<2$ if $k=1$. Then
\begin{equation*}
  \int_{-\infty}^{\infty}\operatorname{sinc}^{k}(x)p(x)\, dx=\Pi^{*k}(0)\cdot\int_{-T/2}^{T/2}p(x)\, dx
\end{equation*}
\end{corollary}

A few identities are then immediate, most notably for $k=1$ and $k=2$,
\begin{equation}
  \label{eq:14}
  \int_{-\infty}^{\infty}\operatorname*{sinc}(x)f(x)\, dx=\int_{-\infty}^{\infty} \operatorname{sinc}^{2}(x)f(x)\, dx=\int_{-1/2}^{1/2}f(x)\, dx
\end{equation}
for $f\in L^1(\mathbb{T})$. It is not a coincidence that the two cardinal sine integrals yield the same value; 
this follows from the fact that $\Pi(0)=\Pi^{*2}(0)=1$ and both $\Pi$ and $\Pi^{*2}$ are continuous at zero.

The identities~\eqref{eq:14} further reduce to the well-known Dirichlet and Fej\'{e}r integrals~\cite{stein_fourier_2011} when $f\equiv 1$,
\begin{align*}
  \int_{-\infty}^{\infty}\frac{\sin\pi x}{\pi x}\, dx&= 
  \int_{-\infty}^{\infty}\left(\frac{\sin\pi x}{\pi x} \right)^{2}\, dx=1.
\end{align*}
More interestingly, when $k=3$, we obtain
\begin{align}
  \int_{-\infty}^{\infty}\operatorname{sinc}^{3}(x)f(x)\, dx&=\hat{f}(0)\Pi^{*3}(0)+\hat{f}(-1)\Pi^{*3}(-1)+\hat{f}(1)\Pi^{*3}(1) \notag \\
  &=\int_{-1/2}^{1/2}f(x)\, d x-\frac{1}{2}\int_{-1/2}^{1/2}f(x)\sin^{2}(\pi x)\, d x \label{eq:43}
\end{align}
and, when $k=4$, 
\begin{align}
  \int_{-\infty}^{\infty}\operatorname{sinc}^{4}(x)f(x)\, dx&=\hat{f}(0)\Pi^{*4}(0)+\hat{f}(-1)\Pi^{*4}(-1)+\hat{f}(1)\Pi^{*4}(1) \notag \\
  &=\int_{-1/2}^{1/2}f(x)\, d x-\frac{2}{3}\int_{-1/2}^{1/2}f(x)\sin^{2}(\pi x)\, d x. \label{eq:38}
\end{align}
The identity~\eqref{eq:38} was previously derived in~\cite{jolany_extension_2018} using complex analytic methods, while~\eqref{eq:43} 
is new, as the results in~\cite{jolany_extension_2018} did not extend to odd-valued integers $k$.


\section{The Poisson Summation Formula}
\label{sec:proofs-theorems}

One of the crucial ingredients in the proof of Theorem~\ref{thm:2} is a version of the Poisson summation formula~\cite[Eq.~(13.4)]{zygmund_trigonometric_2002}. It later appeared explicitly in~\cite[Proposition~1]{baillie_surprising_2008} for functions of compact support. We present the theorem here, along with a clearer proof, which follows~\cite[p.~68]{zygmund_trigonometric_2002}. 

\begin{theorem}[Poisson summation formula]
  \label{thm:3}
Let $g\in L^{1}(\mathbb{R})$, and suppose that there exists some $A>0$ such that 
$\supp g\subset[-A,A]$. Further, let $g$ be of bounded variation in neighborhoods of all $n\in\mathbb{Z}$ 
with $\lvert{n}\rvert\leq A$. Then
  \begin{equation}
    \label{eq:1}
    \sum_{m\in\mathbb{Z}}\hat{g}(m)=\sum_{\substack{n\in\mathbb{Z} \\ \lvert n\rvert\leq A}}\frac{g(n^{+})+g(n^{-})}{2}
  \end{equation}
\end{theorem}

The following identity is immediate since $g$ is of compact support.
  \begin{equation}
    \label{eq:7}
    \sum_{m=-\infty}^{\infty}\hat{g}(m+\xi)=\sum_{\substack{n\in\mathbb{Z} \\ \lvert n\rvert\leq A}}\frac{g(n^{+})+g(n^{-})}{2}e^{-2\pi i n\xi}
  \end{equation}


\begin{proof}[Proof of Theorem~\ref{thm:3}]
Following~\cite[p.~68]{zygmund_trigonometric_2002}, we define a periodic function $G$ on $\mathbb{R}$ as follows, 
  \begin{equation*}
    G(x):=\sum_{k=-\infty}^{\infty}g(x+k).
  \end{equation*}
Since $\supp g=[-A, A]$ the sum on the right will be finite for any fixed $x\in\mathbb{R}$. Furthermore, 
this will be the case also when $x$ varies in $[-\frac{1}{2}, \frac{1}{2}]$, so we can write, for suitable $K\in\mathbb{N}$,
\begin{align*}
\int_{-\frac{1}{2}}^{\frac{1}{2}}\lvert{G(x)}\rvert \,dx&\leq
\sum_{k=-K}^{K}\int_{-\frac{1}{2}}^{\frac{1}{2}}\lvert g(x+k)\rvert \,dx \\
&=\sum_{k=-K}^{K}\int_{k-\frac{1}{2}}^{k+\frac{1}{2}}\lvert g(x)\rvert \,dx \\
&=\int_{-\infty}^{\infty}\lvert g(x)\rvert \,dx
\end{align*}
where we have again used the boundedness of the support of $g$. The last integral is finite since $g\in L^{1}(\mathbb{R})$ 
and we conclude that $G\in L^{1}(\mathbb{T})$. 

Since $g$ is of bounded variation in neighborhoods of those  $n\in\mathbb{Z}$ with $\lvert n\rvert\leq A$, we deduce that $G$ is of bounded variation in a neighborhood of $x=0$. Therefore, we can apply the Dirichlet-Jordan test for Fourier series~\cite[p.~406]{titchmarsh_theory_1939} to deduce that the Fourier series expansion of $G$ converges in a neighborhood of $x=0$ as follows, 
\begin{equation}
  \label{eq:10}
  \frac{G(x^{+})+G(x^{-})}{2}=\sum_{m=-\infty}^{\infty}\hat{G}(m)e^{2\pi im x}.
\end{equation}
A direct calculation yields
  \begin{align*}
    \hat{G}(m)&=\int_{-\frac{1}{2}}^{\frac{1}{2}}G(x)e^{-2\pi im x}\,dx 
\\
    &=\sum_{k=-N}^{N}\int_{-\frac{1}{2}}^{\frac{1}{2}}g(x+k)e^{-2\pi im x}\,dx 
\\
    &=\int_{-\infty}^{\infty}g(x)e^{-2\pi im x}\,dx=\hat{g}(m)
  \end{align*}
Setting $x=0$ in~\eqref{eq:10} then establishes~\eqref{eq:1}.
\end{proof}

The above version of the Poisson summation formula may be aplied to functions that are neither of Schwartz class nor
of moderate decay (both need to be continuous as required, e.g., in~\cite{stein_fourier_2011}). 

For example, neither $\Pi$ nor $\operatorname*{sinc}$ are of Schwartz class or moderate decay, yet we can still obtain a meaningful Poisson summation formula for $\Pi$ by applying Theorem~\ref{thm:3}. In particular, using  $\mathcal{F}[\Pi(\pi(\cdot))](\xi)=\frac{1}{\pi}\operatorname*{sinc\bigl(\frac{\xi}{\pi}\bigr)}$, we obtain
\begin{equation*}
  \frac{1}{\pi}\sum_{n\in\mathbb{Z}}\frac{\sin n}{n}=\sum_{m\in\mathbb{Z}}\Pi(\pi m)=\Pi(0)=1.
\end{equation*}

We are now ready to prove Theorem~\ref{thm:2}.

\begin{proof}[Proof of Theorem~\ref{thm:2}]
Since $f$ is periodic with period $1$, we can write
\begin{align}
\int_{-\infty}^{\infty}\hat{g}(\xi)f(\xi)\,d\xi&=\lim_{M\to\infty}\sum_{m=-M}^{M}\int_{m-\frac{1}{2}}^{m+\frac{1}{2}}\hat{g}(\xi)f(\xi)\,d\xi\notag \\
      &=\lim_{M\to\infty}\sum_{m=-M}^{M}\int_{-\frac{1}{2}}^{\frac{1}{2}}\hat{g}(\xi+m)f(\xi+m)\,d\xi\notag \\
      &=\lim_{M\to\infty}\int_{-\frac{1}{2}}^{\frac{1}{2}}\sum_{m=-M}^{M}\hat{g}(\xi+m)f(\xi)\,d\xi\notag
\end{align}
The Poisson summation formula~\eqref{eq:7} guarantees that the series converges and is bounded, so that by the dominated convergence theorem 
the limit and the integral can be exchanged.   %
Moreover, 
  \begin{align}
       \int_{-\infty}^{\infty}\hat{g}(\xi)f(\xi)\,d\xi &=\int_{-\frac{1}{2}}^{\frac{1}{2}}\lim_{M\to\infty}\sum_{m=-M}^{M}\hat{g}(\xi+m)f(\xi)\,d\xi\notag \\
      &=\int_{-\frac{1}{2}}^{\frac{1}{2}}\sum_{\substack{n\in\mathbb{Z} \\ \lvert n\rvert\leq A}}\frac{g(n^{-})+g(n^{+})}{2}e^{-2\pi i n \xi}f(\xi)\,d\xi\notag \\
    &=\sum_{\substack{n\in\mathbb{Z} \\ \lvert n\rvert\leq A}}\frac{g(n^{-})+g(n^{+})}{2}\hat{f}(n),\notag
  \end{align}
completing the proof.
\end{proof}

\bibliographystyle{plain}
\bibliography{manuscript-lbk}
\end{document}